\documentclass[a4paper,11pt]{newLag}
\makeatletter

\definecolor{sina}{rgb}{0.0,0.0,0.0} 

\definecolor{sinadelete}{rgb}{0.0,0.0,0.0} 

\definecolor{mat}{rgb}{0.0,0.0,0.0} 

\definecolor{sofya}{rgb}{0.0,0.0,0.0} 
\def\sofya{\textcolor{sofya}}

\definecolor{sigrid}{rgb}{0.0,0.0,0.0} 

\definecolor{ssigrid}{rgb}{0.0,0.0,0.0} 

\def\rodrigo{\textcolor{black}}
\def\flora{\textcolor{black}}

\usepackage{amsmath}
\usepackage{mathtools}
\usepackage{amsthm}
\usepackage{amsfonts}  
\usepackage{amssymb}
\DeclareOldFontCommand{\rm}{\normalfont\rmfamily}{\mathrm}
\DeclareOldFontCommand{\sf}{\normalfont\sffamily}{\mathsf}
\DeclareOldFontCommand{\tt}{\normalfont\ttfamily}{\mathtt}
\DeclareOldFontCommand{\bf}{\normalfont\bfhttps://www.overleaf.com/project/62022d046a9b74af4550fc2aseries}{\mathbf}
\DeclareOldFontCommand{\it}{\normalfont\itshape}{\mathit}
\DeclareOldFontCommand{\sl}{\normalfont\slshape}{\@nomath\sl}
\DeclareOldFontCommand{\sc}{\normalfont\scshape}{\@nomath\sc}
\makeatother
\usepackage{leftidx}
\usepackage{psfrag}  
\usepackage[utf8]{inputenc}
\usepackage{bm}
\usepackage{ulem}
\usepackage{graphicx}
\usepackage{etaremune}
\usepackage{graphicx}
\usepackage{booktabs}
\usepackage{adjustbox}
\usepackage[numbib,notlof,notlot,nottoc]{tocbibind}
\usepackage{hyperref}
\usepackage{todonotes}
\usepackage{filecontents}
\usepackage[pagewise,mathlines]{lineno} 

\allowdisplaybreaks

\let\textquotedbl="
\usepackage{url, eurosym}

\usepackage{epstopdf}

\usepackage{tikz} 
\tikzstyle{every pin}=[
					   pin edge = black,					   
					   font=\footnotesize]
\tikzstyle{block} = [draw, fill=blue!20, rectangle, minimum height=3em, minimum width=6em]
\tikzstyle{sum} = [draw, fill=blue!20, circle, node distance=1cm]
\tikzstyle{input} = [coordinate]
\tikzstyle{output} = [coordinate]
\tikzstyle{pinstyle} = [pin edge={to-,thin,black}]

\usepackage{etoolbox} 

\newcommand*\linenomathpatch[1]{%
  \cspreto{#1}{\linenomath}%
  \cspreto{#1*}{\linenomath}%
  \csappto{end#1}{\endlinenomath}%
  \csappto{end#1*}{\endlinenomath}%
}
\newcommand*\linenomathpatchAMS[1]{%
  \cspreto{#1}{\linenomathAMS}%
  \cspreto{#1*}{\linenomathAMS}%
  \csappto{end#1}{\endlinenomath}%
  \csappto{end#1*}{\endlinenomath}%
}

\expandafter\ifx\linenomath\linenomathWithnumbers
  \let\linenomathAMS\linenomathWithnumbers
  \patchcmd\linenomathAMS{\advance\postdisplaypenalty\linenopenalty}{}{}{}
\else
  \let\linenomathAMS\linenomathNonumbers
\fi

\linenomathpatch{equation}
\linenomathpatchAMS{gather}
\linenomathpatchAMS{multline}
\linenomathpatchAMS{align}
\linenomathpatchAMS{alignat}
\linenomathpatchAMS{flalign}

\usepackage{pgffor} 
	\usetikzlibrary{arrows.meta}
	\usetikzlibrary{positioning}
	\usetikzlibrary{bending}
	\usetikzlibrary{shadows}
	\usetikzlibrary{shadows.blur}
	\usetikzlibrary{scopes} 
	\usetikzlibrary{patterns} 
	\usetikzlibrary{fadings} 
	\usetikzlibrary{decorations.pathmorphing} 
	\usetikzlibrary{calc} 
	\usetikzlibrary{datavisualization} 
	\usetikzlibrary{datavisualization.formats.functions} 
	\usetikzlibrary{intersections} 
	\usetikzlibrary{plotmarks}
	\usetikzlibrary{matrix} 
	\usetikzlibrary{arrows} 
	\usetikzlibrary{calc} 
	\usetikzlibrary{shapes,arrows}
	\usetikzlibrary{angles,quotes}
	\usetikzlibrary{cd}

\usepackage{pgfplots}
\pgfplotsset{compat=newest}
\usetikzlibrary{plotmarks}
\usetikzlibrary{arrows.meta}
\usepgfplotslibrary{patchplots}
\usepackage{grffile}
\usepackage{setspace}

\pgfplotsset{plot coordinates/math parser=false}
\newlength\figureheight
\newlength\figurewidth

\ltdsetup{\today}{}{}

\usepackage{caption} 
\usepackage{subcaption} 
\usepackage{tabularx} 
\usepackage[font=small]{caption}
\usepackage[utf8]{inputenc}
\usepackage{paralist}
\usepackage{sectsty}

\sectionfont{\fontsize{12}{15}\selectfont}

\renewcommand\refname{Bibliography concerning the state of the art, the research objectives, and the work program}

\newtheorem{theorem}{Theorem}[section]

\newtheorem{proposition}[theorem]{Proposition}
\newtheorem{example}[theorem]{Example}
\newtheorem{remark}[theorem]{Remark}

\title{A new Lagrangian approach to control affine systems with a quadratic
Lagrange term}
\date{\today}

\author
{Sigrid Leyendecker\footnote{Friedrich-Alexander-Universität Erlangen-Nürnberg (FAU), Institute of Applied Dynamics (LTD), Immerwahrstrasse 1, 91058 Erlangen, Germany. Email: \href{mailto:sigrid.leyendecker@fau.de}{sigrid.leyendecker@fau.de}}\ 
\qquad 
 Sofya Maslovskaya\footnote{\textit{First author}. Universität Paderborn (UPB), Numerical Mathematics and Control (NMC), Warburger Straße 100, 33098 Paderborn, Germany. Email: \href{mailto:sofya.maslovskaya@upb.de}{sofya.maslovskaya@upb.de}}\ \thanks{The work of this author has been supported by Deutsche Forschungsgemeinschaft (DFG), Grant No. OB 368/5-1, AOBJ: 692093}
\\
Sina Ober-Bl\"obaum\footnote{  Universität Paderborn (UPB), Numerical Mathematics and Control (NMC), Warburger Straße 100, 33098 Paderborn, Germany. Email: \href{mailto:sinaober@math.uni-paderborn.de}{sinaober@math.uni-paderborn.de}}\
\qquad
Rodrigo T.~Sato Mart{\'\i}n de Almagro\footnote{\textit{First author}. Friedrich-Alexander-Universität Erlangen-Nürnberg (FAU), Institute of Applied Dynamics (LTD), Immerwahrstrasse 1, 91058 Erlangen, Germany. Email: \href{mailto:rodrigo.t.sato@fau.de}{rodrigo.t.sato@fau.de}}\
 \\
Flóra Orsolya Szemenyei\footnote{\textit{First author}, \textit{corresponding author}. Friedrich-Alexander-Universität Erlangen-Nürnberg (FAU), Institute of Applied Dynamics (LTD), Immerwahrstrasse 1, 91058 Erlangen, Germany. Email: \href{mailto:flora.szemenyei@fau.de}{flora.szemenyei@fau.de}}\ \thanks{The work of this author has been supported by Deutsche Forschungsgemeinschaft (DFG), Grant No. LE 1841/12-1, AOBJ: 692092.}
}

\begin{document}
\maketitle

\footnotetext{{\textit{ Math Subject Classifications. Primary:}} 65K10, 49M25.   {\textit{ Secondary:}} 65K15. }


\footnotetext{This article has been  accepted for publication in a revised form in Journal of Computational Dynamics \url{https://www.aimsciences.org/jcd}. This version is free to download for private research and study only.  Not for redistribution, re-sale or use in derivative works.}

\section*{Abstract}
In this work, we consider optimal control problems for mechanical systems with fixed initial and free final
state and a quadratic Lagrange term. Specifically, the dynamics is described by a second order ODE
containing an affine control term. Classically, Pontryagin’s maximum principle gives necessary optimality
conditions for the optimal control problem. For smooth problems, alternatively, a variational approach
based on an augmented objective can be followed. Here, we propose a new Lagrangian approach leading to
equivalent necessary optimality conditions in the form of Euler-Lagrange equations. Thus, the differential
geometric structure (similar to classical Lagrangian dynamics) can be exploited in the framework of
optimal control problems. In particular, the formulation enables the symplectic discretisation of the
optimal control problem via variational integrators in a straightforward way.

\section{Introduction}
The optimal control of mechanical problems is omnipresent in our technically affected daily living as well as in many scientific questions. {These problems have a rather rich geometric structure. The underlying uncontrolled system frequently lives on a manifold $\mathcal{M}$ that admits a natural symplectic structure, \flora{such as in} the case of Hamiltonian or regular Lagrangian mechanical systems. Moreover, the associated optimal control problem evolves on $T^* \mathcal{M}$, which always admits a symplectic structure. This hierarchy of structures is even more critical in the fully-actuated problem, where one naturally arrives at higher-order mechanical problems \cite{deLeonRodrigues85}, \cite{Colombo2016}}, \cite{treanta14}.\\
The symplectic structure of optimal control problems also plays a major role in analysing numerical methods for the approximation of solutions.
In principle, numerical solution methods for optimal control problems can be classified into direct and indirect methods (see \cite{formalskii10}, \cite{betts2010}).
%
The main difference between the two approaches is the order in which the discretisation and the optimisation steps take place. 
{The indirect approach (first optimise, then discretise) provides necessary optimality conditions given by the adjoint differential equation whereas the direct approach (first discretise, then optimise) yields a discrete version of the adjoint differential equation through the derivation of Karush-Kuhn-Tucker equations (\cite{betts2010,gerdts2003}). The relation between direct and indirect approaches is given by symplectic methods, i.e.~the discrete state and adjoint systems derived by the direct approach is a symplectic discretisation of the continuous state and adjoint system derived in the indirect approach. First works analysing the relationship of direct and indirect approaches for Runge-Kutta methods are e.g.~\cite{hager00}, \cite{bonnans04} and \cite{sanz-serna2015} (see also references therein). Starting directly with a symplectic method in the direct approach for the optimal control of mechanical systems provides a double symplectic scheme (symplectic in the state and symplectic in the state-adjoint equations). This was proven for a particular class of symplectic methods (\cite{ober-blobaum2008, campos15}) by exploiting the hierachy of symplectic structures mentioned above.}   

{The symplectic nature of optimal control problems has motivated many different works over the last years not only with respect to the relation between direct and indirect approaches. Further topics of investigation are e.g.~the geometric interpretation of adjoint systems, the formulation of concise Lagrangians and corresponding variational principles for optimal control problems based on e.g.~higher order Lagrangians or generating functions and associated consistent symplectic discretisation schemes for optimal control problems (see \rodrigo{\cite{deleon2007, leok_tran22, ColomboDeDiegoZucalli}}).}

As a new contribution to this research field, we provide a systematic approach based on a new Lagrangian
formulation and a variational principle involving both, state and adjoint variables of the optimal control problem. More concretely, for a specific class of mechanical optimal control problems (namely control affine systems with a quadratic Lagrange term), we show that the Euler-Lagrange equations of the new Lagrangian provide the classical necessary optimality conditions. Since the Lagrangian is regular, a Hamiltonian can be easily derived based on the Legendre transformation. Furthermore, we investigate invariances of the Lagrangian and associated symmetries in the optimal control problem leading to conserved quantities by Noether's theorem. \rodrigo{Moreover, this new Lagrangian approach opens up the possibility of directly applying variational integrators to the problems considered without needing to involve second-order Lagrangians.}

\section{Preliminaries, notations}

{In this chapter we introduce basic concepts of Hamiltonian and Lagrangian mechanics, optimal control problems and different approaches to derive necessary optimality conditions.}

\subsection{Lagrangian and Hamiltonian mechanics}
\label{ssec:LHmechanics}

{A Lagrangian mechanical system is defined by a pair $(\mathcal{Q}, L)$, where $\mathcal{Q}$ is the configuration manifold of  $\dim \mathcal{Q} = d$ and $L: T \mathcal{Q} \to \mathbb{R}$ is the Lagrangian function of the system. Here, $T \mathcal{Q}$ denotes the tangent bundle of $\mathcal{Q}$, also known as velocity phase space\rodrigo{, with canonical projection $\tau_\mathcal{Q}: T\mathcal{Q} \to \mathcal{Q}$}. Throughout we assume local coordinates $(q^1,...,q^d) = q$ on $\mathcal{Q}$ and adapted coordinates on $T\mathcal{Q}$, $(q^1,...,q^d, \dot{q}^1,...,\dot{q}^d) = (q,\dot{q})$.}
In mechanics, we usually restrict our attention to Lagrangians of the form kinetic energy $T(q, \dot{q})$ minus potential energy $V(q)$. Hamilton's principle requires the action integral
{\begin{equation*}
\int_0^T L(q(t),\dot{q}(t)) \, \mathrm{d} t
\end{equation*}
to be stationary over physical trajectories $q \in C^k([0,T],\mathcal{Q})$, $k \geq 2$, subject to fixed boundary conditions} resulting in the Euler-Lagrange equations of motion. The connection between the Lagrangian and Hamiltonian settings can be achieved by the Legendre transformation. The dynamics in the Hamiltonian setting is defined on the cotangent bundle $\rodrigo{T}^* \mathcal{Q}$ with coordinates 
{$(q^1,...,q^d, p_1,...,p_d) = (q,p)$}
and the Hamiltonian $H:\rodrigo{T}^*\mathcal{Q} \rightarrow \mathbb{R}$ representing the system's energy. Figure \ref{fig:LHmech} depicts the connections. For further details we refer to \cite{arnold78, marsden94}.\\
In this paper we denote the Lagrangian and Hamiltonian of mechanical systems with $L$ and $H$, respectively. Next, we will introduce the Hamiltonian of Pontryagin’s maximum principle and the augmented objective and Lagrangian, which are all denoted by calligraphic letters, namely, $\mathcal{H}$ for the Hamiltonian and $\mathcal{J}$ and $\mathcal{L}$ for the objective and Lagrangian respectively.


\begin{figure}[ht]
	\centering
	\begin{tikzpicture}[thick, scale=1]

	\node[rectangle, align = center, minimum width = 5cm, minimum height = 2cm] (LS) at (-5.0, 0.0) {Lagrangian system \\ $\big(q(t), \dot{q}(t)\big)\in \rodrigo{T}\mathcal{Q}$};

	\node[rectangle, align = center, minimum width = 5cm, minimum height = 2cm] (EL) at (-5.0, -1.2) {Euler-Lagrange equations \\ $-\frac{d}{dt} \partial_{\dot{q}}L +\partial_{q} L=0$};

	\node[rectangle, align = center, minimum width = 5cm, minimum height = 2cm] (HS) at (5.0, 0.0) {Hamiltonian system \\ $\big(q(t), p(t)\big)\in \rodrigo{T}^*\mathcal{Q}$};

	\node[rectangle, align = center, minimum width = 5cm, minimum height = 2cm] (HE) at (5.0, -1.2) {Hamilton's equations \\ $\dot{q} = \partial_{p}H,\quad \dot{p} = -\partial_{q}H$};

	\draw[-latex, very thick] (EL.5) -- (HE.175);
	\draw[-latex, very thick] (HE.185) -- (EL.355);
	
	\node[rectangle, align = center, minimum width = 3cm, minimum height = 2cm] (LT) at (0.0, 0) {Legendre transformation \\ $\mathbb{F}L:\rodrigo{T}\mathcal{Q}\rightarrow\rodrigo{T}^*\mathcal{Q}$\\ {$(q,\dot{q}) \mapsto (q, p = \partial_{\dot{q}}L)$}};

	\end{tikzpicture}
	\caption{Schematic representation of the connection between Lagrangian and Hamiltonian mechanics. Here, $(q(t), \dot{q}(t))$ stands for a solution of the Euler-Lagrange equations corresponding to the Lagrangian $L$, and $(q(t), p(t))$ for a solution of Hamilton's equations corresponding to the Hamiltonian $H$.
 The Legendre transformation $\mathbb{F}L$ connects the Lagrangian and Hamiltonian sides of mechanics and therefore the Euler-Lagrange equations and Hamilton's equations.}
	\label{fig:LHmech}
	\end{figure}
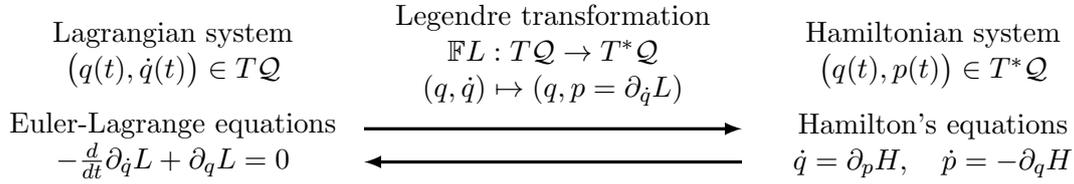

 \subsection{Optimal control problem}
 
\rodrigo{Consider a fibre bundle $\pi: \mathcal{E} \to \mathcal{M}$ with typical fibre $\mathcal{N}$. Locally, a curve on $\mathcal{E}$ parametrized by time $t \in [0, T]$ can be split into $(x(t),u(t))$, where $x(t) \in \mathcal{M}$ is referred to as the state of the system, and $u(t) \in \mathcal{N}$ is referred to as a the control. In the context of control theory, $(\mathcal{E}, \pi, \mathcal{M})$ is frequently a vector bundle, i.e. $\mathcal{N}$ is a vector space, and we assume so throughout. We also assume that $\dim \mathcal{N} \leq \dim \mathcal{M}$. Thus a}n optimal control problem (OCP) on a smooth manifold $\mathcal{M}$ \rodrigo{takes the following local form}
\begin{equation} \label{eq:OCP}
	\begin{aligned}
		& \;\;\;\qquad\underset{u}{\text{min}}  & & J\big(x,u\big) =  \phi\big( x(T)\big)+ \int_0^T \ell\big(x(t), u(t) \big)~dt& \\
		& \text{subject to} & & x(0) \in \mathcal{M}_0,&\\
		& & &x(T) \in \mathcal{M}_T,&\\
		& & & \dot{x}(t)=f(x(t),u(t)),&\\
	\end{aligned}
\end{equation}
The final time $T$ is assumed to be finite, i.e.\ $0 < T < +\infty$ \rodrigo{and fixed throughout}. 
Initial and final constraints on the state are defined by smooth manifolds $\mathcal{M}_0$ and $\mathcal{M}_T$. \rodrigo{The dynamics of the system} is defined by a {continuously differentiable} \rodrigo{vector field along $\pi$,} $f: \mathcal{E} \rightarrow T\mathcal{M}$. The cost functional $J$ consists of an integral over the Lagrange term $\ell$, also called running cost, and the Mayer term \rodrigo{or terminal cost} $\phi$ for the end condition on the state. Both maps $\ell: \mathcal{E} \rightarrow \mathbb{R}$ and \flora{$\phi: \mathcal{M} \rightarrow \mathbb{R}$} are assumed to be {continuously differentiable.} 
In general, the control $u$ is assumed to be a bounded measurable function.
More precisely, we have the following general function spaces,
\begin{gather*}
x\in W^{1,\infty} ([0,T],\mathcal{M}), \,%
 				u\in L^{\infty} ([0,T],\mathcal{N}),\,%
 				\ell \in C^1 (\rodrigo{\mathcal{E}}, \mathbb{R}),\\
 				\phi \in C^1(\mathcal{M},\mathbb{R}),\,%
 				f \in C^1(\rodrigo{\mathcal{E}}, T\mathcal{M})%
 				, \, J \in C^1 (\rodrigo{W^{1,\infty} ([0,T],\mathcal{M}) \times L^{\infty} ([0,T]},\mathcal{N}), \mathbb{R}).
\end{gather*}
 However, 
 later we will see that we need stronger assumptions for our purposes.

 \subsection{Necessary optimality conditions}

Pontryagin's maximum principle (PMP), see \cite{pontryagin1964}, yields necessary optimality conditions for the OCP in \eqref{eq:OCP} in form of a generalised Hamiltonian system and can be stated as follows. If 
{$u$ denotes an optimal solution curve}
of \eqref{eq:OCP} and 
{$x$ denotes}
the associated optimal trajectory, then there exist 
{an adjoint curve $\lambda$ and a multiplier $\lambda_0 \in \mathbb{R}_{-}$ such that $(\lambda,\lambda_0) \neq 0$, and $\rodrigo{(x,\lambda,u,\lambda_0)}$ satisfies}
a generalised Hamiltonian system associated with the control Hamiltonian 
{$\mathcal{H} : \rodrigo{T^* \mathcal{M} \oplus_{\mathcal{M}}  \mathcal{E}} \times \mathbb{R} \to \mathbb{R}$,}
defined by 
\begin{align*}
    \mathcal{H}\rodrigo{(x, \lambda, u, \lambda_0)} = \langle \lambda, f(x,u) \rangle + \lambda_0 \, \ell(x,u),
\end{align*}
 where \rodrigo{$\oplus_{\mathcal{M}}$ denotes the Whitney sum of vector bundles over $\mathcal{M}$ \cite{CrampinPirani86, Saunders89} and} {$\langle \cdot, \cdot \rangle: T^* \mathcal{M} \times T \mathcal{M} \to \mathbb{R}$} is the {canonical} pairing 
 {of covectors and vectors.}
 In local coordinates, it takes the form of {the} standard Euclidean product between vectors. The generalised Hamiltonian system associated to $\mathcal{H}$ provides us with 
differential state equation,
differential adjoint equation and optimality condition
in the following form \begin{equation} \label{eq:PMP}
\begin{cases}
    \dot{x}= \hphantom{-}\partial_{\lambda}\mathcal{H}\rodrigo{(x, \lambda, u, \lambda_0)}, \\ 
    \dot{\lambda}= -\partial_x\mathcal{H}\rodrigo{(x, \lambda, u, \lambda_0)}, \\ 
     0 = \hphantom{-}\partial_u\mathcal{H}\rodrigo{(x, \lambda, u, \lambda_0)}, 
\end{cases}
\end{equation}
where $ (x,\lambda) \in W^{1,\infty}([0,T],T^*\mathcal{M}), \,%
		u\in L^{\infty} ([0,T],\mathcal{N}),\,
	$ and $\mathcal{H}\in C^1(\rodrigo{T^* \mathcal{M} \oplus_{\mathcal{M}} \mathcal{E}} \times \mathbb{R}, \mathbb{R})$.
In addition, the transversality conditions define the relation between state and adjoint variables at the initial and final time as follows
\begin{equation} \label{eq:trasvers}
\lambda(0) \perp T_{x(0)}\mathcal{M}_{0},  \qquad \lambda(T) - \lambda_0 \, \partial_x\phi\big( x(T)\big)  \perp T_{x(T)}\mathcal{M}_{T}.
\end{equation}
A map $t\mapsto(x(t), \lambda(t))$ satisfying the conditions of the PMP is called an extremal. An extremal is called normal if the associated $\lambda_0$ satisfies $\lambda_0 < 0$ and abnormal if $\lambda_0 = 0$. Notice that \eqref{eq:PMP} is invariant under the rescaling of $(\lambda_0, \lambda(t))$ by any positive constant and in {the} case of a normal extremal it is usual to fix $\lambda_0 = -1$. 
Two special cases of interest are the case {where} the initial and final conditions in \eqref{eq:OCP} are fixed, i.e., $x(0) = x_0, \ x(T) = x_T$, and the case {where} the initial condition is fixed, $x(0) = x_0$, and the final state is free, $x(T) \in \mathcal{M}$. 
In the first case, the transversality conditions \eqref{eq:trasvers} are empty because $T_{x(0)}\mathcal{M}_0 = T_{x(T)}\mathcal{M}_T = \{0\}$. In the second case, \eqref{eq:trasvers} implies $\lambda(T) = \lambda_0 \, {\partial_x}\phi\big(x(T)\big)$ and  only $\lambda_0 = -1$ is possible to ensure {$(\lambda,\lambda_0) \neq 0$. Therefore, there are no abnormal extremals in this case.} \\
Let us now consider in more detail the variational augmented objective approach as an alternative technique to derive optimality conditions. Compared to PMP, the variational approach 
{assumes more regularity}
\cite{clarke1990}.  
{It} relies on the introduction of an augmented objective $\mathcal{J}$, 
constructed in such a way that it appends the equality and inequality constraints from~\eqref{eq:OCP} to the objective functional via Lagrange multipliers. The resulting cost can be written in the following form.
\begin{align}
	\mathcal{J}\big(x, \lambda, u\big) &=  \phi\big(x(T)\big)+ \int_0^T \mathcal{L}(x(t), \dot{x}(t),\lambda(t),u(t)) \, dt\nonumber\\
 &= \phi\big(x(T)\big)+ \int_0^T \Big[ \ell\big(x, u \big) + \rodrigo{\left\langle \lambda, \big(\dot{x} - f(x,u) \right\rangle} \Big] ~dt,\label{eqn:augObj}
\end{align}
where the integral term of the augmented objective $\mathcal{J}$ in (\ref{eqn:augObj}), \rodrigo{$\mathcal{L}: T \mathcal{M} \oplus_\mathcal{M} T^* \mathcal{M} \oplus_\mathcal{M} \mathcal{E} \to \mathbb{R}$,} is also called augmented Lagrangian of~\eqref{eq:OCP} (not to \rodrigo{be confused} with the Lagrangian term $\ell$).
 Necessary optimality conditions are derived via the calculus of variations requiring the stationarity of the augmented objective, i.e.,
\begin{align}
\label{stationarityJ}
    \delta \mathcal{J}(x,\lambda, u)=0.
\end{align}
 This variational approach gives equivalent necessary optimality conditions as PMP {for normal cases}, see e.g. \cite[Sec.~3.4,~Sec.~4.1]{liberzon12}, and thus equation \eqref{stationarityJ} 
 again leads to \eqref{eq:PMP} and \eqref{eq:trasvers}.
 However, this approach requires stronger smoothness assumptions on the state and adjoint variables compared to PMP. More precisely, we have to assume 
 \begin{gather} 
     u \in C^1([0,T], \mathcal{N}), \,
			(x,\lambda) \in C^2([0,T],T^*\mathcal{M}),\label{eq:variational.classes1} \\
			f \in C^1(\rodrigo{\mathcal{E}}, T\mathcal{M}),\,%
			\ell \in C^1 (\rodrigo{\mathcal{E}}, \mathbb{R}),\,%
			\phi \in C^1(\mathcal{M}, \mathbb{R}),\label{eq:variational.classes2} 
 \end{gather}
 and hence
 \begin{align*}
     \mathcal{L} \in C^1 (\rodrigo{T\mathcal{M} \oplus_\mathcal{M} T^*\mathcal{M} \oplus_\mathcal{M} \mathcal{E}}, \mathbb{R}) \quad \text{and} \quad \mathcal{J} \in C^1( \rodrigo{C^2([0,T],T^*\mathcal{M}) \times C^1([0,T], \mathcal{N})}, \mathbb{R}),
 \end{align*}
 For a general optimal control problem such as \eqref{eq:OCP}, the Lagrangian $\mathcal{L}$ is \rodrigo{singular} as no time derivative $\dot{\lambda}$ appears. Therefore, a Legendre transformation to the Hamiltonian setting is not well defined. There exist some approaches where a generalised Legendre transformation is defined to overcome this issue. For example, Gotay and Nester's version of Dirac's constraint algorithm \rodrigo{\cite{gotay1979, ColomboDeDiegoZucalli}} provides the correct setting for non-regular Lagrangians due to constraints. However, the algorithm is difficult to use in practice. \\
 The goal of this work is
to derive a new uniform Lagrangian on a proper tangent space
which then directly yields a symplectic integration scheme for
state and adjoint variable if we apply a variational integration scheme {\cite{marsden01}} and from which easily a Hamiltonian can be derived via Legendre
transformation. A key ingredient of the following is the introduction of a new Hamiltonian $\tilde{\mathcal{H}}$ and Lagrangian $\tilde{\mathcal{L}}$ for optimal control problems. Explicit dependence on the control $u$ is denoted by a superscript, e.g.~$\tilde{\mathcal{H}}^u$.

\section{Construction of a new Lagrangian for a control affine system with a specific objective}

In the present work \rodrigo{we focus on a particular} case of the OCP in \eqref{eq:OCP}, where the state equation is described by a second order ordinary differential equation (ODE) as a dynamical constraint. \rodrigo{This is common in the case of controlled mechanical systems. Thus, let us assume $\mathcal{M} = T \mathcal{Q}$, where $\mathcal{Q}$ is the configuration manifold as in Section~\ref{ssec:LHmechanics}. Moreover, 
we assume that $\mathcal{Q}$ is a Riemannian manifold with metric $\mathrm{g}_\mathcal{Q}$. Further, we assume that $\mathcal{E}$ is an anchored vector bundle over $\mathcal{Q}$, $\pi: \mathcal{E} \to \mathcal{Q}$, with injective linear anchor $\rho: \mathcal{E} \to T \mathcal{Q}$, $\tau_\mathcal{Q} \circ 
\rho = \pi$ \cite{PopescuAnchored}. Then the anchor is locally represented by functions $\rho_{\alpha}^i(q)$, $i = 1,...,\dim \mathcal{Q}$, $\alpha = 1, ..., \dim \mathcal{N}$, and the metric  pulls back to a metric $\mathrm{g} \equiv \mathrm{g}_\mathcal{E}$ on $\mathcal{E}$. This lets us consider under-actuated as well as fully-actuated control-affine mechanical systems, whose dynamics take the local form
\begin{equation}
    \label{eq:SODE}
    \ddot{q}^i = f^i(q,\dot{q}) + \rho^i_{\alpha}(q) u^{\alpha}\,.
\end{equation}
One can interpret this second order ODE (SODE) as a constraint on $T^{(2)}\mathcal{Q} \oplus_{\mathcal{Q}} \mathcal{E}$, where $T^{(k)} \mathcal{Q}$ denotes the $k$-th order tangent bundle of $\mathcal{Q}$ (see \cite{deLeonRodrigues85}). Moreover, it can be rewritten as an explicit system of first-order ODEs:
\begin{equation}\label{eq:system_SODE}
	\begin{bmatrix*}\dot{q}\\\dot{v}\end{bmatrix*} = \begin{bmatrix*}v\\f(q, v)+\rho(q) u\end{bmatrix*}.
\end{equation}
Here, the drift term $f(q,\dot{q})$ can be interpreted as a vector of generalised forces and for each $\alpha$ the components of the vertical lift of the anchor can be interpreted as control fields. Thus, we consider the following problem:
\begin{equation} \label{eq:eqample.OCP.}
	\begin{aligned}
		& \;\;\;\qquad\underset{u}{\text{min}} & & J\big(q,u\big) =  \phi\big(q(T), \dot{q}(T)\big) + \int_0^T \frac{1}{2} \,\mathrm{g}(q(t))(u(t),u(t))~dt& \\
		& \text{subject to} & & q(0) = q^0,&\\
		& & &\dot{q}(0) = \dot{q}^0,&\\
		& & & \ddot{q}=f(q,\dot{q})+\rho(q)u,&\\
	\end{aligned}
\end{equation}}
where $q\in C^3 ([0,T],\mathcal{Q}),\,%
	u \in C^1([0,T],\mathcal{N})$, \,%
	and\rodrigo{, with some abuse of notation,} $\phi \in C^1( T\mathcal{Q}, \mathbb{R})$. \rodrigo{The components $\mathrm{g}_{\,\alpha \beta}(q),\, f^i(q,v)$ and $\rho^i_{\alpha}(q)$ are also assumed to be $C^1$}, and thus, \sofya{the conditions in \eqref{eq:variational.classes1}-\eqref{eq:variational.classes2} are satisfied and the variational approach can be applied.}

\rodrigo{
\begin{remark}
    We would like to highlight that the approach presented in the following for the construction of a new Lagrangian does not rely on this particular structure of the problem, namely, quadratic Lagrange terms and control-affine systems. Rather, we choose these systems to showcase the approach for simplicity. It is sufficient to be able to solve for the controls uniquely in terms of the adjoints, but we intend to explore this in more detail in the future.
\end{remark}}

\subsection{Pontryagin's maximum principle for the control affine system}

We first apply Pontryagin's maximum principle to \eqref{eq:eqample.OCP.}. Since the final condition in \eqref{eq:eqample.OCP.} is free, we have that $\lambda_0=-1$ and the normal Hamiltonian from PMP is defined as follows
\begin{align}
    \mathcal{H}(q,v, \lambda_q, \lambda_v, u)&= \rodrigo{\left\langle (\lambda_q,\lambda_v), (v, f(q,v) + \rho(q)u) \right\rangle - \frac{1}{2} \mathrm{g}(q)(u,u)\nonumber}\\
    &= \lambda_q^\top v+ \lambda_v^\top \rodrigo{(f(q,v) + \rho(q)u) -\frac{1}{2} u^{\top} \mathrm{g}(q)\, u},\label{normalhamiltonian}
\end{align}
where $\lambda_q$ and $\lambda_v$ are the adjoints related to $q$ and $v$, respectively. \rodrigo{In the second line, we introduce matrix notation and we are abusing our notation by identifying the metric with its associated matrix.}
For \eqref{normalhamiltonian} we obtain the following theorem applying PMP.
\begin{theorem}
   Considering the optimal control problem \eqref{eq:eqample.OCP.} the following equations hold
   \begin{subequations}
   \label{eq:eqample.OCP.eqs}
\begin{align}
\dot{\lambda}_q&=\rodrigo{\frac{1}{2}\left(\frac{\partial}{\partial q}\mathrm{g}(q)(u,u)\right)^{\top}-\left( \frac{\partial}{\partial q}f(q,v) + \frac{\partial}{\partial q}\rho(q) u \right)^\top\lambda_v},\label{eq:eqample.OCP.lambda_q}\\
 \dot{\lambda}_v &=-\lambda_q\rodrigo{-\frac{\partial}{\partial v}f(q,v)^\top\lambda_v},\label{eq:eqample.OCP.lambda_v}\\
 u&=\rodrigo{\mathrm{g}^{-1}(q) \,\rho(q)^\top \lambda_v}.\label{eq:eqample.OCP.u}
\end{align}
\end{subequations}
Moreover, for the final time one has 
\begin{align}
\label{finalconstraint}
     \lambda_q(T) =-\frac{\partial}{\partial q}\phi(q(T), v(T))^{\top},\qquad \lambda_v(T) =-\frac{\partial}{\partial v}\phi(q(T), v(T)^{\top}. 
\end{align}
\end{theorem}

\begin{proof}
Using Pontryagin's maximum principle for \eqref{eq:eqample.OCP.}, we obtain the following equations
\begin{align*}
    \begin{bmatrix}
			\dot{q}\\ \dot{v}
		\end{bmatrix}= 
   \begin{bmatrix}
			\frac{\partial \mathcal{H}}{\partial \lambda_q}\\[4pt] \frac{\partial \mathcal{H}}{\partial \lambda_v}
		\end{bmatrix} = 
  \begin{bmatrix}
			v\\ \rodrigo{f(q,v)+\rho(q)u}
		\end{bmatrix},\quad 
    \begin{bmatrix}
    			\dot{\lambda}_q\\ \dot{\lambda}_v
		\end{bmatrix}=- 
  \begin{bmatrix}
			\frac{\partial \mathcal{H}}{\partial q}\\[4pt] \frac{\partial \mathcal{H}}{\partial v}
		\end{bmatrix} =\rodrigo{ 
  \begin{bmatrix}
			(\frac{\partial}{\partial q}\mathrm{g}(q)(u,u))^{\top} -(\frac{\partial}{\partial q}f(q,v) + \frac{\partial}{\partial q}\rho(q) u )^\top\lambda_v\\ -\lambda_q -\frac{\partial}{\partial v}f(q,v)^\top\lambda_v
   \end{bmatrix},}
\end{align*}
and 
\begin{equation*}
    0 = \frac{\partial}{\partial u}\mathcal{H}=\rodrigo{u^{\top} \mathrm{g}(q) - \lambda_v^{\top} \rho(q)}\,.
\end{equation*} 
\rodrigo{Since $g$ is a Riemannian metric, its matrix is invertible, which allows us to isolate $u$ and obtain Equation~\eqref{eq:eqample.OCP.u}.} From Pontryagin's maximum principle one can also derive condition\rodrigo{s} \eqref{finalconstraint} directly, which finishes the proof.
\end{proof}
 \begin{remark}
     Notice that \rodrigo{eliminating the controls $u$ and the multiplier $\lambda_q$ using \eqref{eq:eqample.OCP.u} and \eqref{eq:eqample.OCP.lambda_v} respectively, the system of equations reduces to the system of SODEs
     \begin{subequations}
\label{eq:eqample.constraint.2order}
\begin{align} 
         \ddot{q} &= f(q,\dot{q}) + \rho(q) \mathrm{g}^{-1}(q) \,\rho(q)^\top \lambda_v\,,\label{eq:eqample.constraint.2order.q}\\ \ddot{\lambda}_v &= \frac{1}{2}\left(\frac{\partial}{\partial q}\mathrm{g}^{-1}(q)(\rho(q)^\top \lambda_v, \rho(q)^\top \lambda_v)\right)^{\top}\nonumber\\
         &- \frac{d}{d t}\left(\frac{\partial}{\partial \dot{q}}f(q,\dot{q})^\top\lambda_v \right) + \left(\frac{\partial}{\partial q}f(q,\dot{q}) + \frac{\partial}{\partial q}\rho(q) \mathrm{g}^{-1}(q) \,\rho(q)^\top \lambda_v\right)^\top\lambda_v.
     \label{eq:eqample.constraint.2order.lambda_v}\end{align}
     \end{subequations}
     where we have used that for every differentiable and invertible matrix function $(A^{-1})' = -A^{-1}A' A^{-1}$.}
 \end{remark}

\subsection{Augmented objective approach} We now consider the variational approach for problem \eqref{eq:eqample.OCP.}.
Note that the classical approaches work with the system of first order ODEs, {which is \eqref{eq:system_SODE} in our case. The corresponding Lagrange multipliers are $\lambda_q,\: \lambda_v$.} 
With that, we can
define the augmented Lagrangian for \eqref{eq:eqample.OCP.} as follows
\begin{align}
\label{lagr}
\mathcal{L}(q,v, \dot{q}, \dot{v}, \lambda_q, \lambda_v, u) &= \rodrigo{\frac{1}{2} \mathrm{g}(q)(u,u) + \left\langle(\lambda_q,\lambda_v), (\dot{q}-v, \dot{v}-f(q,v) -\rho(q) u )\right\rangle}\nonumber\\
&=\rodrigo{\frac{1}{2} u^{\top} \mathrm{g}(q)\,u + \lambda_q^\top\big(\dot{q} - v \big)+ \lambda_v^\top ( \dot{v} -f(q,v) - \rho(q) u)},
\end{align}
where
\begin{align*}
   &(q,v,\lambda_q,\lambda_v) \in C^2([0,T], T^*(T\mathcal{Q})),\\
    &\mathcal{L}\in C^1 (\rodrigo{(T(T\mathcal{Q}) \oplus_{T\mathcal{Q}} T^*(T\mathcal{Q})) \oplus_{\mathcal{Q}} \mathcal{E}}, \mathbb{R}).
\end{align*}

With this Lagrangian, we can derive the following theorem considering the augmented objective approach.

\begin{theorem}
   The augmented objective 
   \begin{align*}
\mathcal{J}(q, v,\lambda_q, \lambda_v, u)&=
   \phi\big(q(T),v(T)\big)+ \int_0^T \mathcal{L}(q,v, \dot{q}, \dot{v}, \lambda_q, \lambda_v, u) \, dt \\
   &=\phi\big(q(T),v(T)\big)+ \int_0^T \rodrigo{\left[\frac{1}{2} u^{\top} g(q) \, u + \lambda_q^\top\big(\dot{q} - v \big)+ \lambda_v^\top ( \dot{v} - f(q,v) - \rho(q) u) \right]}\, dt
   \end{align*}
   is stationary for the respective variations, if the constraints in \eqref{finalconstraint} and \rodrigo{Equations~\eqref{eq:system_SODE} and \eqref{eq:eqample.OCP.eqs} hold}.
   \end{theorem}

\begin{proof}
After \rodrigo{taking variations of $\mathcal{J}$ and applying integration by parts on the terms involving temporal derivatives, we arrive at}
\rodrigo{
\begin{align*}
    \delta\mathcal{J} &=\left(\lambda_q(T) + \frac{\partial}{\partial q}\phi(q(T), v(T))^{\top}
  \right)^{\top}\delta q(T) + \left(  \lambda_v(T) +\frac{\partial}{\partial v}\phi(q(T), v(T))^{\top}\right)^{\top} \delta v(T)\\
    &+ \int_0^T \left[\, \vphantom{\left(\frac{\partial}{\partial q}\right)^{\top}}\delta \lambda_q^{\top} (\dot{q} - v) + \delta \lambda_v^{\top} \left( \dot{v} -( f(q,v) + \rho(q) u)\right) + (\mathrm{g}(q) u - \rho(q)^{\top}\lambda_v)^{\top} \delta u \right.\\
    &\quad\qquad+ \left(\frac{1}{2}\left(\frac{\partial}{\partial q}\mathrm{g}(q)(u,u)\right)^{\top}-\dot{\lambda}_q -\left(\frac{\partial}{\partial q}f(q,v) + \frac{\partial}{\partial q}\rho(q) u\right)^\top\lambda_v\right)^{\top} \delta q\\
    &\quad\qquad\left.+ \left(-\dot{\lambda}_v - \lambda_q - \frac{\partial}{\partial v}f(q,v)^\top \lambda_v\right)^{\top} \delta v\,\right] d t.
\end{align*}
}
    Since the variations are independent, \rodrigo{under the conditions of the theorem, the stationarity of the action follows.}
 \end{proof}
\begin{remark}
    Notice that the derivatives $\dot{\lambda}_q$ and $\dot{\lambda}_v$ do not appear in $\mathcal{L}$ in \eqref{lagr}, and thus, the derivatives of $\mathcal{L}$ with respect to $\dot{\lambda}_q,\: \dot{\lambda}_v$ vanish, and therefore the Jacobian of $\mathcal{L}$ \rodrigo{with respect 
 to the derivatives} is singular.
\end{remark}

\subsection{New Lagrangian and Euler-Lagrange equations}\label{ssec:new_lagrangian}


We \rodrigo{define a} new Lagrangian $\tilde{\mathcal{L}}^u$ \rodrigo{by} appending the \rodrigo{SODE \eqref{eq:SODE}} to the cost \rodrigo{instead of \eqref{eq:system_SODE}} and doing integration by part of the term $\lambda^\top \ddot{q}$ as follows:
\rodrigo{
\begin{align}
\label{eq:new.lagrange.}
   & \int_0^T \left[ \frac{1}{2}\,\mathrm{g}(q(t))(u(t),u(t))+\lambda(t)^\top \left(\vphantom{\sum}\ddot{q}(t)-f(q(t),\dot{q}(t))-\rho(q) u(t)\right)\: dt\right]\nonumber\\
    &\quad = \int_0^T \left[\frac{1}{2}\,\mathrm{g}(q(t))(u(t),u(t)) - \dot{\lambda}(t)^\top \dot{q}(t)-{\lambda}(t)^\top \left(\vphantom{\sum}f(q(t),\dot{q}(t)) + \rho(q(t)) u(t)\right)\right] \: dt+ [\lambda(t)^\top \dot{q}(t)]_0^T\notag\\
    &\quad=: \int_0^T \tilde{\mathcal{L}}^u(q, {\lambda}, \dot{q}, \dot{\lambda},u) \: dt+[\lambda(t)^\top \dot{q}(t)]_0^T,
\end{align}}
where
\begin{align*}
(q,{\lambda})\in C^3 ([0,T],T^*\mathcal{Q}),
\end{align*}
and thus 
\begin{align}
\label{eq:actual.new.lagrange.}
\tilde{\mathcal{L}}^u\in C^1(\rodrigo{T(T^*\mathcal{Q})\oplus_{\mathcal{Q}} \mathcal{E}}, \mathbb{R})
\end{align}
which we will use as a new Lagrangian (see Section~\ref{ssec:geometric_setting} for the geometric interpretation of the new Lagrangian). 
Consider variations of the new augmented objective
\begin{equation*}
    \tilde{\mathcal{J}}^u(q,\lambda,u) =  \phi\big(q(T), \dot{q}(T)\big) +\lambda(T)\dot{q}(T)-\lambda(0)\dot{q}(0)+\int_0^T \rodrigo{\left[  
    \frac{1}{2} \mathrm{g}(q)(u,u) -\dot{\lambda}^\top  \dot{q} - \lambda^\top \left( \vphantom{\sum}f(q,\dot{q}) + \rho(q) u\right) \right]}
    dt.  
\end{equation*}
  \sofya{We are in the setting of \eqref{eq:variational.classes1}-\eqref{eq:variational.classes2} with the newly defined $\tilde{\mathcal{J}}^u(q,\lambda,u)$ and the variational approach can be used.} 
The application of \rodrigo{integration by parts leads to 
\begin{align*}
0=\delta \tilde{\mathcal{J}}^u
&=\left( - \dot{\lambda}(T)^{\top} - \lambda(T)^{\top} \frac{\partial}{\partial \dot{q}}f(q(T),\dot{q}(T)) + \frac{\partial}{\partial q}\phi(q(T), \dot{q}(T))
  \right) \delta q(T)\\
  &+ \left(  \lambda(T)^{\top} +\frac{\partial}{\partial \dot{q}}\phi(q(T), \dot{q}(T))\right)^{\top} \delta \dot{q}(T)\\
    &+ \int_0^T \left[\, \left(\ddot{\lambda} + \frac{1}{2}\left(\frac{\partial}{\partial q}\mathrm{g}(q)(u,u)\right)^{\top} -  \left(\frac{\partial}{\partial q}f(q,\dot{q}) + \frac{\partial}{\partial q}\rho(q) u\right)^\top\lambda + \frac{d}{d t} \left( \left(\frac{\partial}{\partial \dot{q}}f(q,\dot{q}) \right)^\top\lambda \right) \right)^{\top} \delta q\right.\\
    &\quad\qquad\left.+ \vphantom{\left(\frac{\partial}{\partial q}\right)^{\top}}\delta \lambda^{\top} \left(\vphantom{\sum} \ddot{q} - f(q,\dot{q}) - \rho(q) u \right) + \left(\vphantom{sum}\mathrm{g}(q) u - \rho(q)^{\top}\lambda\right)^{\top} \delta u \,\right] d t.
\end{align*}}
\rodrigo{where we have used the fact that $\delta q(0)=\delta \dot{q}(0)=0$. Thus,} we deduce 
\rodrigo{
\begin{subequations}
\label{endcond}
\begin{align}
    \dot{\lambda}(T) + \frac{\partial}{\partial \dot{q}}f(q(T),\dot{q}(T))^{\top} \lambda(T) &=  \frac{\partial}{\partial q}\phi(q(T), \dot{q}(T))^{\top} \label{endcond.dlambda}\\
    \lambda(T) &= -\frac{\partial}{\partial \dot{q}}\phi(q(T), \dot{q}(T))^{\top}\label{endcond.lambda}
\end{align}
\end{subequations}
\begin{remark}
    Notice that \eqref{endcond.dlambda} is nothing but the boundary condition for $\lambda_q(T)$ in \eqref{finalconstraint}. It suffices to make the identification $\lambda \equiv \lambda_v$ and eliminating $\lambda_q$ using \eqref{eq:eqample.OCP.lambda_v} evaluated at $T$. Similarly, we see that the resulting equation obtained from the variation with respect to $q$ is \eqref{eq:eqample.OCP.lambda_q} after the same substitution.
\end{remark}}
 
The optimal control \rodrigo{\eqref{eq:eqample.OCP.u}} associated with the new Lagrangian \eqref{eq:actual.new.lagrange.} can be derived in the same way as for the augmented objective, i.e. via a variational principle. Substituting the \rodrigo{optimal control} in $\tilde{\mathcal{L}}^u$ \rodrigo{in terms} of $\lambda$, we define $\tilde{\mathcal{L}}: T (T\rodrigo{^*}\mathcal{\mathcal{Q}}) \to \mathbb{R}$, the new Lagrangian without dependence on the control
\rodrigo{\begin{align}
\tilde{\mathcal{L}}(q, \lambda, \dot{q}, \dot{\lambda}) &=  -\dot{\lambda}^\top \dot{q} - \lambda^\top f(q,\dot{q}) - \frac{1}{2} \mathrm{g}^{-1}(q)(\rho(q)^{\top} \lambda, \rho(q)^{\top} \lambda)\nonumber\\
&=: -\dot{\lambda}^\top \dot{q} - \lambda^\top f(q,\dot{q}) - \frac{1}{2} b(q)(\lambda, \lambda).\label{eq:new.lagrange.no.input}
\end{align}
where we have introduced the symmetric bilinear form $b: T^*Q \times T^*Q \to \mathbb{R}$ for notational simplicity. Notice that only in the case where $\dim \mathcal{Q} = \dim \mathcal{N}$ will this be a Riemannian metric, in which case, $b \sim \mathrm{g}_{\mathcal{Q}}^{-1}$.}

\rodrigo{If $\mathcal{Y} = T^* \mathcal{Q}$ and 
$y(t) = (q(t), \lambda(t)) \in \mathcal{Y}$ for $t \in [0,T]$, we see that the new Lagrangian
$\tilde{\mathcal{L}}: T\mathcal{Y} \to \mathbb{R}$ is indeed a Lagrangian in the mechanical sense (Section ~\ref{ssec:LHmechanics}) and we may compute its Euler-Lagrange equations
\begin{align*}
    \frac{\partial \tilde{\mathcal{L}}}{\partial y} (y,\dot{y})-\frac{d}{dt}\left( \frac{\partial \tilde{\mathcal{L}}}{\partial \dot{y}} (y,\dot{y}) \right)=0.
\end{align*}
This results in the following theorem.}
\begin{theorem} 
The Euler-Lagrange equations for the Lagrangian $ \tilde{\mathcal{L}}$ provide the same necessary conditions and boundary constraints as the ones in \eqref{eq:eqample.constraint.2order} of the original minimisation problem \eqref{eq:eqample.OCP.} \rodrigo{under the identification $\lambda = \lambda_v$, i.e.,
\begin{subequations}
\label{eq:EL_new.}
\begin{align} 
         \ddot{q} &= f(q,\dot{q}) + b(q) \lambda\,,\label{eq:eqample.constraint.2order.q}\\ \ddot{\lambda} &= \frac{1}{2}\left( 
 \frac{\partial}{\partial q} \left(b(q)(\lambda, \lambda)\right)\right)^{\top} + \left(\frac{\partial}{\partial q}f(q,\dot{q})\right)^\top\lambda - \frac{d}{d t}\left(\frac{\partial}{\partial \dot{q}}f(q,\dot{q})^\top\lambda \right).
     \label{eq:eqample.constraint.2order.lambda_v}\end{align}
     \end{subequations}}
and 
\rodrigo{
\begin{align*}
    \dot{\lambda}(T) + \frac{\partial}{\partial \dot{q}}f(q(T),\dot{q}(T))^{\top} \lambda(T) &=  \frac{\partial}{\partial q}\phi(q(T), \dot{q}(T))^{\top}\\
    \lambda(T) &= -\frac{\partial}{\partial \dot{q}}\phi(q(T), \dot{q}(T))^{\top}
\end{align*}}
\end{theorem}

\begin{proof}
\rodrigo{A straightforward computation leads us to \eqref{eq:EL_new.},
which corresponds to the state and adjoint equations derived in \eqref{eq:eqample.constraint.2order} under the definitions and identifications above mentioned. The boundary conditions were already obtained in \eqref{endcond}.}
\end{proof}
\rodrigo{Moreover, as can be readily checked, this new Lagrangian is hyperregular. Thus, the Legendre transformation is bijective and it lets us formulate a} Hamiltonian counterpart as depicted in Figure~\ref{fig:LHmech}. 

\sofya{
\begin{remark}
   Now that \eqref{eq:eqample.OCP.} is reformulated in a form of a regular Lagrangian variational problem, we can apply variational integrators to discretise it. This leads to a symplectic discretisation of the corresponding Euler-Lagrange equations \cite{marsden01}. The discrete problem can then be solved numerically. This can be regarded as a new numerical approach to treat optimal control problems and will be investigated in future publications.
\end{remark}
}

\subsection{New Hamiltonian system}\label{ssec:new_hamiltonian}

Let us now consider the Hamiltonian system of the new, regular Lagrangian $\tilde{\mathcal{L}}$. The Legendre transformation,
$\mathbb{F}L:(y,\dot{y}) \mapsto \left(y,{p_y} = \frac{\partial \tilde{\mathcal{L}}}{\partial \dot{y}}(y,\dot{y})\right)$,
for the new Lagrangian provides the conjugate momenta  

\begin{align*}
		p_y=
			 \frac{\partial \tilde{\mathcal{L}}}{\partial \dot{y}}, \qquad \quad\begin{bmatrix}
			p_q\\[0.5em]p_{\lambda}
		\end{bmatrix} =
		\begin{bmatrix}
			 -\dot{\lambda} \sofya{-\frac{\partial}{\partial \dot{q}}f(q, \dot q)^\top\lambda}\\[0.5em] -\dot{q}
		\end{bmatrix},	
\end{align*}
and therefore invertible. 
The new Hamiltonian can be computed as
\begin{align*}
\tilde{\mathcal{H}}: T^* \mathcal{Y} \rightarrow \mathbb{R}, \quad 
\tilde{\mathcal{H}}(y, p_y)&=\tilde{\mathcal{H}}\big(q, \lambda, p_q, p_{\lambda}\big)  = \rodrigo{\left\langle (p_q, p_{\lambda}), (\dot{q},\dot{\lambda})\right\rangle_{\mathcal{Y}}}
 -\tilde{\mathcal{L}}(q, \lambda, \dot{q}, \dot{\lambda}) \notag\\
 &=p_q^\top \dot{q}+p_{\lambda}^\top \dot{\lambda}+ \sofya{\dot{\lambda}^\top \dot{q} + \lambda^\top f(q,\dot{q}) + \frac{1}{2} b(q)(\lambda, \lambda)} \\
 &= -p_{\lambda}^\top p_q+ \sofya{\lambda^\top \tilde f(q,p_{\lambda}) + \frac{1}{2} b(q)(\lambda, \lambda),}
\end{align*}
\sofya{where $\tilde f(q,p_{\lambda}) = f(q,\dot{q}(p_{\lambda}))$. We also define $\tilde\phi(q, p_{\lambda}) = \phi(q, \dot{q}(p_{\lambda}))$.} 
\begin{theorem} \label{Hamiltonsystem}
   Hamilton's equations for $\tilde{\mathcal{H}}$ \sofya{in the variables $y = (q, \lambda)$ and $p_y = (p_q, p_{\lambda})$ are given by} 
   \begin{equation} \label{eq:Ham.sys.}
	\begin{aligned}
		&\dot{y} = 
  \begin{bmatrix}
			-p_{\lambda}\\[0.5em] -p_q \sofya{+ \frac{\partial}{\partial p_{\lambda}} \tilde f(q,p_{\lambda})^\top\lambda}
		\end{bmatrix},\ 
		&\dot{p}_y 
  = 
  \begin{bmatrix}
			- \frac{\partial}{\partial q}\sofya{ \tilde f(q,p_{\lambda})^\top \lambda- \frac{1}{2}\left( 
 \frac{\partial}{\partial q} \left(b(q)(\lambda, \lambda)\right)\right)^{\top}} \\[0.5em] \sofya{- \tilde f(q,p_{\lambda}) - b(q) \lambda}
		\end{bmatrix}
	\end{aligned}
\end{equation}
   \sofya{and are equivalent to \eqref{eq:EL_new.}.} 
   Moreover, one has the following boundary conditions
   \begin{align}
    \label{Hambc}
       \dot{q}(0)&=-p_{\lambda}(0), \quad  \dot{\lambda}(0)=-p_{q}(0) \sofya{ + \frac{\partial}{\partial p_{\lambda}} \tilde f(q(0),p_{\lambda}(0))^\top\lambda(0)},\notag\\
    \dot{q}(T)&=-p_{\lambda}(T), \quad  \dot{\lambda}(T)=-p_{q}(T) \sofya{ + \frac{\partial}{\partial p_{\lambda}} \tilde f(q(T),p_{\lambda}(T))^\top\lambda(T)},
   \end{align}
\sofya{ and
 \begin{align} \label{Hambc2}
    p_q(T) &= - \frac{\partial}{\partial q}\tilde \phi(q(T), p_{\lambda}(T))^{\top}, \\
    \lambda(T) &= \frac{\partial}{\partial p_{\lambda}}\tilde \phi(q(T), p_{\lambda}(T))^{\top} \label{Hambc2},
\end{align}}

\end{theorem}

\begin{proof}
 With $p_y = (p_q, p_{\lambda})$, the associated Hamilton's \sofya{equations are }

\begin{equation*} 
	\begin{aligned}
		&\dot{y} = 
 \frac{\partial}{\partial p_y}\tilde{\mathcal{H}}=
  \begin{bmatrix}
			-p_{\lambda}\\[0.5em] -p_q \sofya{+ \frac{\partial}{\partial p_{\lambda}} \tilde f(q,p_{\lambda})^\top\lambda}
		\end{bmatrix},\ 
		&\dot{p}_y 
  = - \frac{\partial}{\partial y} \tilde{\mathcal{H}}=
  \begin{bmatrix}
			- \frac{\partial}{\partial q}\sofya{ \tilde f(q,p_{\lambda})^\top \lambda- \frac{1}{2}\left( 
 \frac{\partial}{\partial q} \left(b(q)(\lambda, \lambda)\right)\right)^{\top}} \\[0.5em] \sofya{- \tilde f(q,p_{\lambda}) - b(q) \lambda}
		\end{bmatrix}.
	\end{aligned}
\end{equation*}
which are equivalent to the Euler-Lagrange equations in~\eqref{eq:EL_new.}, since 
\begin{align*}
    \ddot{q} =-\dot{p}_{\lambda} 
    \qquad \text{and} \qquad
     \ddot{\lambda} =-\dot{p}_{q} \sofya{ - \frac{d}{d t}\left(\frac{\partial}{ \partial \dot{q}}f(q,\dot{q})^\top\lambda \right)}. 
\end{align*}
Therefore, they are also equivalent to the conditions in \eqref{eq:eqample.constraint.2order} of the original OCP problem. Using \sofya{the relationship between $\dot y$ and $p_y$ together with \eqref{endcond.dlambda}-\eqref{endcond.lambda}, we obtain \eqref{Hambc}-\eqref{Hambc2}.}
\end{proof}

As it was shown in Theorem \ref{Hamiltonsystem} we have commutation in the following diagram.

\begin{center}
	\includegraphics[scale=.5]{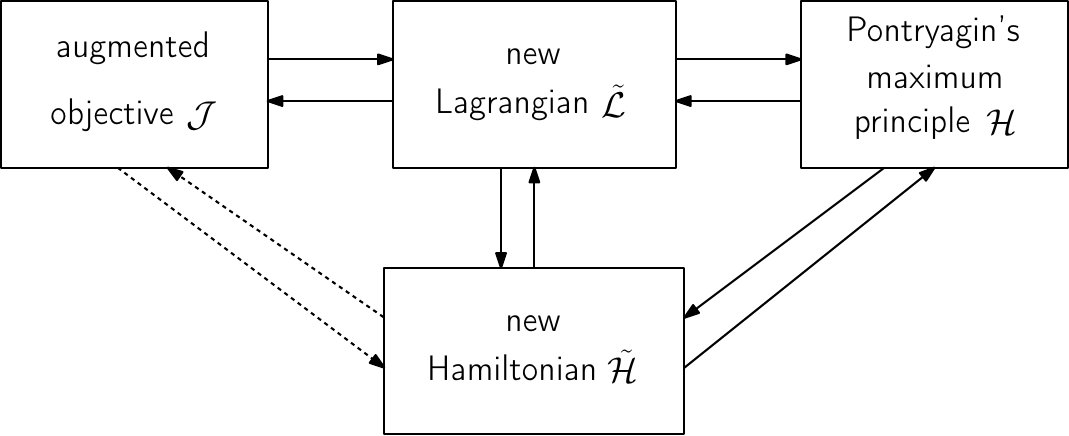}
\end{center}

In summary, for the optimal control problem \eqref{eq:eqample.OCP.} we derived a purely Lagrangian system on the new tangent bundle $\rodrigo{T}\mathcal{Y}$ which involves both, state and adjoint variables of the optimal control problem. Thus we can apply standard techniques as for mechanical Lagrangians. The Euler-Lagrange equations provide the correct state and adjoint dynamics. Furthermore, a purely Hamiltonian description is obtained by applying a Legendre transformation.

 \sofya{
\begin{remark}
  Notice that the relationship between Lagrangian and Hamiltonian formulations will be also preserved on the discrete side if we apply variational integrators. In this case the discrete Legendre transformation transforms from the Lagrangian to the Hamiltonian setting and vice versa. 
\end{remark}
}
\subsection{Geometric setting}\label{ssec:geometric_setting}

For the background in geometric mechanics required in the following section, see e.g. \cite{marsden94, treanta14}.\\

The construction outlined in Section \ref{ssec:new_lagrangian} is intimately related with the so-called Tulczyjew's triple \cite{TulczyjewHam,TulczyjewLag,GrabowskaGrabowski}. This is an isomorphic relation between the three double bundles $T^*(T \mathcal{Q})$, $T(T^*\mathcal{Q})$ and $T^*(T^*\mathcal{Q})$. For our particular case, the relation between the first two is the most important.\\

First, consider the augmented Lagrangian of \eqref{lagr}. This function is defined on \rodrigo{$(T(T\mathcal{Q}) \oplus_{T\mathcal{Q}} T^*(T\mathcal{Q})) \oplus_{\mathcal{Q}} \mathcal{E}$}, and the fact that our dynamics come from a SODE is represented by the constraint $\dot{q} = v$. When moving to \eqref{eq:new.lagrange.}, we are already assuming the restriction to the SODE case and, because of that, both $v$ and $\lambda_q$ disappear from the integrand, and the left-hand side of the first equality is now a function in the more restricted space \rodrigo{$T^{(2)} \mathcal{Q} \oplus_{\mathcal{Q}} T^*\mathcal{Q} \oplus_{\mathcal{Q}} \mathcal{E}$}, and which ends up mapped to \rodrigo{$T(T^*\mathcal{Q}) \oplus_{\mathcal{Q}} \mathcal{E}$} via integration by parts.\\

To better understand this, let us focus on the terms involving $\lambda_q$ and $\lambda_v$ and the right-hand side of \eqref{eq:system_SODE} in \eqref{lagr}. This can be interpreted as the pairing of a vector field $X$ on $T\mathcal{Q}$ (to be more precise a section \rodrigo{$\Gamma(T\mathcal{Q} \oplus_{\mathcal{Q}} \mathcal{E}, T(T \mathcal{Q}) \oplus_{\mathcal{Q}} \mathcal{E})$}) with a $1$-form $\lambda$ on the same space\rodrigo{, so locally}
\begin{align*}
X(q,v,u) &= X_q^i(q,v,u) \, \partial_{q^i} + X_v^i(q,v,u) \,\partial_{v^i}\,, \quad i = 1,..., d\\
\lambda &= \lambda_{q\,i} \,\mathrm{d}{q^i} + \lambda_{v\,i} \,\mathrm{d}{v^i}\,.
\end{align*}
The pairing can thus be regarded as a function on \rodrigo{$T^*(T\mathcal{Q}) \oplus_{\mathcal{Q}} \mathcal{E}$}, where we are choosing local coordinates $(q,v,\lambda_q,\lambda_v,u)$.\\

Tulczyjew's isomorphism $\alpha_{\mathcal{Q}}^{-1}: T^*(T \mathcal{Q}) \to T(T^* \mathcal{Q})$, $(q,v,\lambda_q,\lambda_v) \mapsto (q,\lambda_v,v,\lambda_q)$, allows us \rodrigo{to} reinterpret this pairing as a function on \rodrigo{$T(T^*\mathcal{Q}) \oplus_{\mathcal{Q}} \mathcal{E}$}. If we assume local coordinates $(q,\lambda,v_q,v_{\lambda},u)$ on this latter space, we obtain the identification
\begin{align*}
\lambda_v &= \lambda\\
v &= v_q\\
\lambda_q &= v_{\lambda}
\end{align*}
where now $\lambda$ denotes a form (and the collection of its coordinates by abuse of notation) on $\mathcal{Q}$.\\

Thus, we see that our new Lagrangian
\rodrigo{\begin{equation*}
\tilde{\mathcal{L}}^u(q, \lambda, v_q, v_{\lambda}, u) = \frac{1}{2} \mathrm{g}(q)(u,u) - v_{\lambda}^\top v_q - \lambda^\top [ f(q,v_q) + \rho(q) u ]\,,
\end{equation*}}
is a well-defined function on \rodrigo{$T(T^*\mathcal{Q}) \oplus_{\mathcal{Q}} \mathcal{E}$}, and, in particular, the expression in the last equality of ~\eqref{eq:new.lagrange.} is nothing but the evaluation of this function on the canonical lift of a curve \cite{CrampinPirani86} on \rodrigo{$T^*\mathcal{Q} \oplus_{\mathcal{Q}} \mathcal{E}$ to $T(T^*\mathcal{Q}) \oplus_{\mathcal{Q}} \mathcal{E}$}. More explicitly, if we consider a curve $(q,\lambda) \in C^k([0,T], T^*\mathcal{Q})$ with $k > 0$, its lift to $T(T^*\mathcal{Q})$ gives us the curve $(q, \lambda, \dot{q}, \dot{\lambda}) \in C^{k-1}([0,T], T(T^*\mathcal{Q}))$.

\subsection{Noether's theorem, conserved quantities}

Having the setting as in classical mechanics, in this section we investigate what are
the conserved quantities of a Lagrangian flow corresponding to symmetries (also called invariance)
of the new Lagrangian\rodrigo{s $\tilde{\mathcal{L}}^u$ and} $\tilde{\mathcal{L}}$\rodrigo{, } what is their meaning and their relation to the dynamics and the
OCP \eqref{eq:eqample.OCP.}.

\rodrigo{Let us consider a Lie group $G$ acting on $\mathcal{E}$ with smooth action $\Phi^\mathcal{E}: G \times \mathcal{E} \to \mathcal{E}$. We say that $(\mathcal{E},\pi,\mathcal{Q},\rho)$ is a $G$-equivariant anchored vector bundle if
\begin{enumerate}
\item $\mathcal{E}$ is a $G$-equivariant bundle, i.e. there exist $\Phi^\mathcal{Q} : G \times \mathcal{Q} \to \mathcal{Q}$ such that $\pi \circ \Phi^\mathcal{E} = \Phi^\mathcal{Q}$.
\item \label{itm:anchor_compat} $\rho \circ \Phi^\mathcal{E} = \Phi^{T\mathcal{Q}}$, where $\Phi^{T\mathcal{Q}} : G \times T\mathcal{Q} \to T\mathcal{Q}$ is the tangent lift of $\Phi^\mathcal{Q}$.
\end{enumerate}
Since $\rho$ is assumed to be linear, this implies that for each $g \in G$, $(\Phi_g^{\mathcal{E}},\Phi_g^\mathcal{Q})$ is a vector bundle morphism, i.e. linear in the fibres.
Thus, locally
\begin{align*}
\Phi^{\mathcal{E}}_g(q,u) &= (\Phi^i(g,q),\Psi^{\alpha}_{\beta}(g,q) u^{\beta})\,,\\
\Phi^{T\mathcal{Q}}_g(q,v) &= \left(\Phi^i(g,q),\frac{\partial \Phi^i}{\partial q^j}(g,q) v^{j}\right)\,,
\end{align*}
for some functions $(\Psi_g)^{\alpha}_{\beta}: \mathcal{U} \subset \mathcal{Q} \to \pi^{-1}(\mathcal{U})$. Then condition \ref{itm:anchor_compat} implies that
\begin{equation*}
\rho^i_{\alpha}(\Phi(g,q)) \Psi^{\alpha}_{\beta}(g,q) u^{\beta} = \frac{\partial \Phi^i}{\partial q^j}(g,q) \rho^j_{\alpha}(q) u^{\alpha}\,.
\end{equation*}}
\sofya{From now on we use a simplified notation $\Phi_g(q) = \Phi(g,q)$ and  $\Psi_g(q) = \Psi(g,q)$. Moreover, we denote $T_q \Phi_g = \frac{\partial \Phi}{\partial q}(g,q)$ and the adjoint map $T^*_q \Phi_g$. The same notations are also used for $\Psi_g(q)$.} 
\rodrigo{Let us assume that $\mathcal{E}$ is indeed a $G$-equivariant anchored vector bundle. Consider also the lifts of the action $\Phi^{\mathcal{Q}}$ to $T^* \mathcal{Q}$ and $T(T^*\mathcal{Q})$. }
\sofya{In a local trivialization $(q,\lambda)$ of $T^*\mathcal{Q}$, the action is defined by $T^*\Phi_{g^{-1}}(q, \lambda) = (\Phi_g(q), T^*_{\Phi_g(q)}\Phi_{g^{-1}}(\lambda))$. Then, the action lifted to $T(T^*\mathcal{Q})$ is given by
\begin{equation*}
    (q, \lambda, v_q, v_{\lambda}) \mapsto (\Phi_g(q), T^*_{\Phi_g(q)}\Phi_{g^{-1}}(\lambda), T_q\Phi_g(v_q), T_{\lambda}T^*_{\Phi_g(q)}\Phi_{g^{-1}}(v_{\lambda})).
\end{equation*}}
We shall consider symmetric optimal control problems, in the sense that we assume to have an equivariant right-hand side of state equations 
\begin{align}
\label{equivariance_f}
\sofya{f(\Phi_g(q),T_q\Phi_g\dot{q})+\rho(\Phi_g(q))\rodrigo{\Psi_g(q)\,}u = T_q\Phi_gf(q,\dot{q})+T_q\Phi_g\rho(q)u,}
\end{align}
and invariance of the running cost 
\begin{align} \label{invariance_g}
\sofya{\mathrm{g}(\Phi_g(q))(\rodrigo{\Psi_g(q)\,} u,\rodrigo{\Psi_g(q)\,} u) =\mathrm{g}(q)(u,u).}
\end{align}

Moreover, we define the one parameter group of matrix transformations 
\begin{align*}
   \{{\tilde{\Phi}}_{g_s}:\mathcal{Y}\rightarrow \mathcal{Y}, \quad s \in \mathbb{R}\} \qquad \text{with}\qquad {\tilde{\Phi}}_{g_s}(y) = \sofya{\big(\Phi_{g_s}(q),  T^*_{\Phi_g(q)}\Phi_{g^{-1}}(\lambda) \lambda\big),}
\end{align*}
and the lift of {the} action 
\[
\begin{aligned}
	T{\tilde{\Phi}}_{g_s}(y, \dot{y}) 
	&= \sofya{ (\Phi_g(q), T^*_{\Phi_g(q)}\Phi_{g^{-1}}(\lambda), T_q\Phi_g(\dot q), T_{\lambda}T^*_{\Phi_g(q)}\Phi_{g^{-1}}(\dot{\lambda})). }
\end{aligned}
\]

Applying Noether's theorem, see e.g. \cite{marsden94}, to the new Lagrangian $\tilde{\mathcal{L}}$, one obtains immediately the following statement.
\begin{proposition}
\label{noetherforL}
 Let $G := \{ {\tilde{\Phi}}_{g_s} : \mathcal{Y} \rightarrow\mathcal{Y} ,\: s \in \mathbb{R}\}$ be a one-parameter group of transformations\rodrigo{, i.e.}:
 \begin{itemize}
     \item $ {\tilde{\Phi}}_{g_s} : \mathcal{Y} \mapsto\mathcal{Y}$ is a diffeomorphism for all $s \in \mathbb{R}$,
     \item ${\tilde{\Phi}}_{g_s}(y)$ as a function of $s$ is differentiable for all $y\in \mathcal{Y}$,
     \item ${\tilde{\Phi}}_{g_t}\circ {\tilde{\Phi}}_{g_s}={\tilde{\Phi}}_{g_{t+s}}$ for all $t,\:s \in \mathbb{R} $.
 \end{itemize}
    Moreover, assume that the Lagrangian system $(\mathcal{Y},\tilde{\mathcal{L}})$ is invariant under the action of $G$, i.e.,
    \begin{align*}
       \tilde{\mathcal{L}} \big(  {\tilde{\Phi}}_{g_s}(y), \, T_y{\tilde{\Phi}}_{g_s}(\dot{y}) \big)= \tilde{\mathcal{L}}(y,\dot{y})\quad \text{for all}\quad s\in \mathbb{R}, \: (y,\dot{y})\in T\mathcal{Y}.
    \end{align*}
  Then the momentum map
\begin{align*}
  I(y,\dot{y})= \frac{\partial \tilde{\mathcal{L}}}{\partial \dot{y}}(y,\dot{y}) \cdot 
  \rodrigo{\left.\frac{d}{ds}\right\vert_{s = 0}}\sofya{{\tilde{\Phi}}_{g_s}}(y)
\end{align*}
is a conserved quantity of the motion. 
\end{proposition}
\sofya{Let us check that if an optimal control problem of the form \eqref{eq:eqample.OCP.} is symmetric with respect to the $G$-action in the sense of \eqref{equivariance_f}-\eqref{invariance_g}, then the associated Lagrangian $ \tilde{\mathcal{L}}$ is invariant with respect to $G$ and Proposition~\ref{noetherforL} applies. First notice that \eqref{equivariance_f}-\eqref{invariance_g} imply $b(\Phi_{g_s}q)( T^*_{\Phi_{g_s}(q)}\Phi_{{g_s}^{-1}}(\lambda),  T^*_{\Phi_{g_s}(q)}\Phi_{{g_s}^{-1}}(\lambda)) = b(q)(\lambda, \lambda)$. This follows from
$$
\mathrm{g}^{-1}(\Phi_{g_s}q)(\rho(q)^* \circ T^*_q\Phi_{g_s} \circ T^*_{\Phi_{g_s}(q)}\Phi_{{g_s}^{-1}}(\lambda), \rho(q)^* \circ T^*_q\Phi_{g_s} \circ T^*_{\Phi_{g_s}(q)}\Phi_{{g_s}^{-1}}(\lambda)) = \mathrm{g}^{-1}(q)(\rho(q)^* \lambda, \rho(q)^* \lambda). 
$$
Using the equivariance of $f$ and $\rho$ in \eqref{equivariance_f} and invariance of $\mathrm{g}$ in \eqref{invariance_g} we obtain 
\begin{align*}
    \tilde{\mathcal{L}}\big(T{\tilde{\Phi}}_{g_s}(y, \dot{y})\big)  &= - \langle T^*_{\Phi_{g_s}(q)}\Phi_{{g_s}^{-1}}(\dot \lambda), T_q\Phi_{g_s}(\dot{q}) \rangle  - \langle T^*_{\Phi_{g_s}(q)}\Phi_{{g_s}^{-1}}(\lambda),  T_q\Phi_{g_s}\circ f(q, \dot q)\rangle  - \frac{1}{2}  b(q)(\lambda, \lambda) \\ 
    &= -\langle \dot \lambda, \dot{q} \rangle  - \langle \lambda,   f(q, \dot q)\rangle  - \frac{1}{2}  b(q)(\lambda, \lambda)\\
      &=  \tilde{\mathcal{L}}(y, \dot{y}).
\end{align*}
Thus, we obtained that $\tilde{\mathcal{L}} $ is invariant under $G$.} \\

Let us now consider \sofya{an example of an action of rotation group on a configuration space given by an Euclidean space.}
\begin{example}
\label{exm:SO3}
 \sofya{Let $\mathcal{Q}=\mathbb{R}^3$ and $\mathcal{E}$ be a trivial bundle $\mathcal{E} = \mathcal{Q} \times \mathcal{N}$ with $\mathcal{N} = \mathbb{R}^3$. 
 Let $SO(3)$ act by matrix multiplication on $\mathcal{Q}$ and on $\mathcal{N}$. The optimal control problem is given as follows
 \begin{equation} \label{eq:example.SO(3)}
	\begin{aligned}
		& \underset{u}{\text{min}} & & J\big(q,u\big) =  \int_0^T \frac{1}{2} u^2(t)~dt& \\
		& \text{subject to} & & q(0) = q^0,&\\
		& & &\dot{q}(0) = \dot{q}^0,&\\
		& & & \ddot{q}=f(q)+u.
	\end{aligned}
\end{equation}
We assume that $f(q)$ is equivariant with respect to  $SO(3)$ action, which implies in this case that the OCP is symmetric in the sense of \eqref{equivariance_f}-\eqref{invariance_g}. Consider now a }
%
one-parameter group \sofya{$G$} of rotations around the axis $n\in\mathbb{R}^3$, $\| n\|_2=1$. Its action on $T\mathcal{Y}$ can be written as 
\begin{equation*}
    {\tilde{\Phi}}_{g_s}(y)={\tilde{\Phi}}_{g_s}(q,\lambda)=(\exp(s\hat{n})q,\exp(s\hat{n})\lambda ),
\end{equation*}
where  
\begin{equation*}
    \hat{n}=
\begin{bmatrix}
    0 & -n_3& n_2\\
    n_3 &0& -n_1\\
    -n_2& n_1& 0
\end{bmatrix}.
\end{equation*}
\sofya{Equivariance of $f$ with respect to $G$ can be expressed as follows}
\begin{align}
    \label{equivariancef}
    f(\exp(s\hat{n})q)=\exp(s\hat{n})f(q), \quad s\in \mathbb{R},\: q\in \mathbb{R}^3.
\end{align}
In this case, $\tilde{\mathcal{L}} $ is invariant under the action of $G$, since
\begin{align*}
    \tilde{\mathcal{L}}\big(T{\tilde{\Phi}}_{g_s}(y, \dot{y}))\big) 
    &=-\left(\exp(s\hat{n})\dot{\lambda}\right)^\top\exp(s\hat{n}) \dot{q} - \left(\exp(s\hat{n})\lambda\right)^\top f(\exp(s\hat{n})q) - \frac{1}{2} (\exp(s\hat{n})\lambda)^2 \\ 
    &=  \sofya{-\dot{\lambda}^\top \dot{q} - \lambda^\top f(q) - \frac{1}{2} \lambda^\top \lambda} \\
    &=\tilde{\mathcal{L}}(y,\dot{y}) ,
\end{align*}  
where we used \eqref{equivariancef} and the fact that
\begin{equation*}
    \left(\exp(s\hat{n})a\right)^\top \exp(s\hat{n})b=a^\top b \quad \text{for all}\quad  a,\:b\in \mathbb{R}^3,
\end{equation*}
which can be shown e.g. by using Rodrigues' rotation formula
\begin{equation*}
    g_s=\exp(s\hat{n})=I_3+ \sin(s) \hat{n}+(1-\cos(s))\hat{n}^2.
\end{equation*}

Thus, applying Proposition \ref{noetherforL} we obtain the conserved quantity
\begin{align*}
     I(y,\dot{y})= \frac{\partial \tilde{\mathcal{L}}}{\partial \dot{y}}(y,\dot{y}) \cdot  
 \rodrigo{\left.\frac{d}{ds}\right\vert_{s = 0}}\sofya{{\tilde{\Phi}}_{g_s}}(y)= 
  \begin{bmatrix}
      -\dot{\lambda}\\
       -\dot{q}
\end{bmatrix}^\top
    \begin{bmatrix}
      \hat{n}q\\
       \hat{n}\lambda
  \end{bmatrix}= -\dot{\lambda}^\top(n\times q)-\dot{q}^\top (n\times \lambda)= n^\top (\dot{q}\times \lambda + \dot{\lambda}\times q).
\end{align*}
Note that in our case $\dot{q}=-p_{\lambda}$ and $\dot{\lambda}=-p_q$, see \eqref{eq:Ham.sys.}, and therefore 
\begin{align*}
n^\top (\dot{q}\times \lambda + \dot{\lambda}\times q)= n^\top (q\times p_q+\lambda\times p_{\lambda})
\end{align*}
is a conserved quantity.
\end{example}

\subsection{A mechanical example - Low thrust orbital transfer}
\flora{As \rodrigo{an under-actuated} mechanical example we consider \rodrigo{the following orbital dynamics problem} \cite{ober-blobaum2008a}. A satellite with mass $m$ moves in the gravitational field of the Earth (mass $M$). The
satellite is to be transferred from one circular orbit to one in the same plane with a larger radius, while the
number of revolutions around the Earth during the transfer process is fixed. In polar coordinates $q = ( r, \varphi )$,
the Lagrangian of the system\rodrigo{, $L: T(\mathbb{R}^2\setminus \left\lbrace 0 \right\rbrace) \to \mathbb{R}$,} has the form
\begin{align*}
    L(q,\dot{q})=L(r,\varphi,\dot{r},\dot{\varphi})=\frac{1}{2}m (\dot{r}^2+r^2 \dot{\varphi}^2)+\gamma \frac{Mm}{r},
\end{align*}
with $\gamma$ being the gravitational constant. Assume that the propulsion system continuously exhibits a force $u$
only in the direction of circular motion (i.e. orthogonal to the vector $r$), such that the corresponding Lagrangian control force is \rodrigo{locally} given by \rodrigo{$f_L = m r u \, \mathrm{d}\varphi$. Thus, $\mathcal{N} = \mathbb{R}$ and $\mathcal{E} = \mathbb{R}^2\setminus \left\lbrace 0 \right\rbrace \times \mathbb{R}$. The anchor is locally represented by $\frac{1}{r} \partial_{\varphi} \otimes \mathrm{d} u$ and $\mathrm{g}^{\mathcal{Q}} = \mathrm{d}r \otimes \mathrm{d}r + r^2 \mathrm{d} \varphi \otimes \mathrm{d} \varphi$, so $\mathrm{g} = \mathrm{d} u \otimes \mathrm{d} u$.}\\
We consider the following objective functional and boundary conditions:
\begin{align*}
J(u)&=\int_0^T \rodrigo{\frac{1}{2}} u(t)^2 \: dt,\\
    q(0)&=(r(0),\varphi(0))=(r^0,0),\\
    \dot{q}(0)&=(\dot{r}(0),\dot{\varphi}(0))= (0,\sqrt{\gamma M/(r^0)^3}),\\
     q(T)&=(r(T),\varphi(T))=(r^T,0),\\
     \dot{q}(T)&=(\dot{r}(T),\dot{\varphi}(T))= (0,\sqrt{\gamma M/(r^T)^3}),
\end{align*}
where $T=d\sqrt{\frac{4\pi^2}{8\gamma M}(r^0+r^T)^3}$ with $d$ being a prescribed number
of revolutions around the Earth.\\
We apply the forced Euler-Lagrange equations
\begin{equation*}
  \frac{d}{dt} \frac{\partial}{\partial \dot{q}}L(q(t),\dot{q}(t)) = \frac{\partial}{\partial {q}}L(q(t),\dot{q}(t)) + f_L(q,u)
\end{equation*}
for $L$ obtaining
\begin{align*}
   m\ddot{r}&=mr\dot{\varphi}^2-\frac{\gamma M m}{r^2},\\
  mr^2\ddot{\varphi}&=\rodrigo{-2 m r \dot{r} \dot{\varphi}} + mru.
\end{align*}
This leads to the second order \sofya{control system}
\begin{align*}
    \ddot{q}(t)=( \ddot{r}(t), \ddot{\varphi}(t))=\left(r\dot{\varphi}^2-\frac{\gamma M }{r^2},\rodrigo{-\frac{2 \dot{r} \dot{\varphi}}{r}} + \frac{u}{r}\right),
\end{align*}
which \rodrigo{is of} the form \eqref{eq:SODE}.\\
Let us now consider the following augmented objective with $\lambda=(\rodrigo{\lambda_r,\lambda_\varphi})$:
\begin{align*}
   \mathcal{J}(q,\lambda,u)&= \phi\big(q(T), \dot{q}(T)\big)+ \int_0^T \mathcal{L} (q,\dot{q},\ddot{q},\lambda,u)\: dt\\
    &= \phi\big(q(T), \dot{q}(T)\big)+ \int_0^T \left[ \rodrigo{\frac{1}{2}} u(t)^2+\rodrigo{\lambda_r}\left(\ddot{r}-r\dot{\varphi}^2\rodrigo{+}\frac{\gamma M }{r^2}\right)+\rodrigo{\lambda_\varphi} \left(\ddot{\varphi} \rodrigo{\,+ \frac{2 \dot{r} \dot{\varphi}}{r}} -\frac{u}{r} \right)\right] \: dt    \\
     &= \phi\big(q(T), \dot{q}(T)\big)+ ( \rodrigo{\lambda_r} \dot{r})|_0^T  + ( \rodrigo{\lambda_\varphi} \dot{\varphi})|_0^T\\
     &+\int_0^T \left[ \rodrigo{\frac{1}{2}} u(t)^2- \rodrigo{\dot{\lambda}_r} \dot{r} - \rodrigo{\dot{\lambda}_\varphi} \dot{\varphi} - \rodrigo{\lambda_r} \left( r \dot{\varphi}^2 - \frac{\gamma M }{r^2}\right) - \rodrigo{\lambda_\varphi}\left(\rodrigo{-\frac{2 \dot{r} \dot{\varphi}}{r}} + \frac{u}{r}\right) \right]\: dt.
\end{align*}
\rodrigo{Taking variations} we obtain:
\rodrigo{
 \begin{align*}
     \delta u:& \quad u = \frac{\lambda_\varphi}{r}\\
     \delta r:& \quad \ddot{\lambda}_r = \lambda_r \left( \dot{\varphi}^{2}+\frac{2 \gamma  M}{r^3} \right) + \lambda_{\varphi} \left(\frac{2 \ddot{\varphi}}{r}-\frac{u}{r^{2}}\right) + \dot{\lambda}_{\varphi} \frac{2 \dot{\varphi}}{r}\\
     \delta \varphi:& \quad \ddot{\lambda}_{\varphi} = 2\left[ -\lambda_r \left(\ddot{\varphi} r +\dot{\varphi} \dot{r}\right) +  \lambda_{\varphi} \left(\frac{\ddot{r}}{r}-\frac{\dot{r}^{2}}{r^{2}}\right) -r \dot{\varphi} \dot{\lambda}_r + \dot{\lambda}_{\varphi} \frac{\dot{r}}{r} \right],
 \end{align*}}
and \rodrigo{eliminating $u$} we can define the Lagrangian 
\rodrigo{\begin{equation*}
\tilde{\mathcal{L}}(q, \lambda, \dot{q}, \dot{\lambda}) = - \rodrigo{\dot{\lambda}_r} \dot{r} - \rodrigo{\dot{\lambda}_\varphi} \dot{\varphi} - \rodrigo{\lambda_r} \left(r \dot{\varphi}^2 - \frac{\gamma M }{r^2}\right) + \rodrigo{\lambda_\varphi} \frac{2 \dot{r} \dot{\varphi}}{r} - \frac{1}{2} \frac{\lambda_{\varphi}^2}{r^2}.
\end{equation*}}
\rodrigo{It can be seen that the initial control system is rotationally invariant in the sense of \eqref{equivariance_f}-\eqref{invariance_g}. Moreover, the new Lagrangian is clearly rotationally invariant since $\varphi$ is cyclic. The one-parameter group of transformations associated with this symmetry in polar coordinates is simply angular translations, i.e $\tilde{\Phi}_{g_s}(r,\varphi,\lambda_r,\lambda_{\varphi}) = (r,\varphi + s,\lambda_r,\lambda_{\varphi})$ with infinitesimal generator $\partial_{\varphi}$. Thus, a}pplying Noether's theorem we obtain the conserved quantity
\rodrigo{
\begin{equation*}
     I(q, \lambda, \dot{q}, \dot{\lambda})= \begin{bmatrix}
      -\dot{\lambda}_r + \frac{2 \lambda_{\varphi} \dot{\varphi}}{r}\\
      -\dot{\lambda}_{\varphi} - 2 r \lambda_r \dot{\varphi} + 2 \frac{\lambda_{\varphi} \dot{r}}{r}\\
     -\dot{r}\\
    -\dot{\varphi}
  \end{bmatrix}^\top
  \begin{bmatrix}
      0\\
      1\\
      0\\
      0
  \end{bmatrix} = -\dot{\lambda}_{\varphi} - 2 r \lambda_r \dot{\varphi} + 2 \frac{\lambda_{\varphi} \dot{r}}{r}.
\end{equation*}
}}

\rodrigo{In Cartesian coordinates, the new Lagrangian takes the form
\begin{equation*}
\tilde{\mathcal{L}}(q, \lambda, \dot{q}, \dot{\lambda}) = - \dot{\lambda}_x \dot{x} - \dot{\lambda}_y \dot{y} + \gamma M \frac{x \lambda_x + y \lambda_y}{(x^2+y^2)^{3/2}} - \frac{(x \lambda_y - y \lambda_x)^2}{x^2+y^2}.
\end{equation*}
and the conserved quantity becomes
\begin{equation*}
    I(q, \lambda, \dot{q}, \dot{\lambda}) = - \dot{x} \lambda_y + \dot{y} \lambda_x + x \dot{\lambda}_y - y \dot{\lambda}_x
\end{equation*}
which can be seen to be of the form of that in Example \ref{exm:SO3} with $n = (0,0,1)$.
}

\paragraph{Outlook on future work}
The method considered in the paper is based on the introduction of the new Lagrangian formulation.
\rodrigo{This new formulation can be applied to construct numerical integration methods for smooth optimal control problems applying variational integration techniques, which we will explore in the future.} 
The approach can be {further} generalised 
to optimal control problems with general second order dynamical constraints and with \flora{a} more general Lagrange term. Moreover, different regularity requirements are also an interesting question in future works. It would be of particular {importance} 
to consider 
larger classes of control, such as piecewise continuous or essentially bounded measurable functions. This can be {crucial} 
for applications where it is not possible to have smooth controls. \rodrigo{In regards to numerical integration, we hope that this new approach will let us naturally apply further developments in the field, such as} high-order variational integrators as well as multirate integrators to the optimal control problem. An extension \rodrigo{to} multisymplectic dynamics could also provide an interesting research area. Optimal time problems could be handled via an extended Lagrangian viewpoint. 

\section*{Acknowledgements}
The authors acknowledge the support of Deutsche Forschungsgemeinschaft (DFG) with the projects: LE 1841/12-1, AOBJ: 692092 and OB 368/5-1, AOBJ: 692093.


{\footnotesize{
		\makeatletter
		\renewenvironment{thebibliography}[1]
		{%
			\@mkboth{\MakeUppercase\refname}{\MakeUppercase\refname}%
			\list{\@biblabel{\@arabic\c@enumiv}}%
			{\settowidth\labelwidth{\@biblabel{#1}}%
				\leftmargin\labelwidth
				\advance\leftmargin\labelsep
				\@openbib@code
				\usecounter{enumiv}%
				\let\p@enumiv\@empty
				\renewcommand\theenumiv{\@arabic\c@enumiv}}%
			\sloppy
			\clubpenalty4000
			\@clubpenalty \clubpenalty
			\widowpenalty4000%
			\sfcode`\.\@m}
		{\def\@noitemerr
			{\@latex@warning{Empty `thebibliography' environment}}%
			\endlist}
		\makeatother

		\bibliographystyle{wmaainf}
		
		}}

\begingroup\scriptsize

\makeatletter
\renewcommand\@openbib@code{\itemsep\z@}

\bibliography{bib_OCP}
\endgroup

\end{document}